\documentclass[12pt,a4paper]{amsart}
\usepackage{a4wide}
\usepackage{amsmath, amssymb, amsfonts,enumerate}
\usepackage[all]{xy}
\usepackage{amscd}

\newtheorem{theorem}{Theorem}[section]
\newtheorem{lemma}[theorem]{Lemma}
\newtheorem{corollary}[theorem]{Corollary}
\newtheorem{proposition}[theorem]{Proposition}

\theoremstyle{definition}

\newtheorem*{definition*}{Definition}

\theoremstyle{remark}
\newtheorem{remark}[theorem]{Remark}

\newtheorem{example}[theorem]{Example}

\numberwithin{equation}{section}

\newcommand {\N}{\mathbb{N}} 
\newcommand {\Z}{\mathbb{Z}} 
\newcommand {\R}{\mathbb{R}} 

\newcommand{\FF}{\mathcal{F}}

\newcommand{\TT}{\mathcal{T}}

\DeclareMathOperator{\ent}{ent}
\DeclareMathOperator{\inte}{Int}

\DeclareMathOperator{\adhe}{Adh}

\begin{document}
\title[The Myhill property for cellular automata]{The Myhill property for cellular automata on amenable semigroups}

\author[T.Ceccherini-Silberstein]{Tullio Ceccherini-Silberstein}
\address{Dipartimento di Ingegneria, Universit\`a del Sannio, C.so
Garibaldi 107, 82100 Benevento, Italy}
\email{tceccher@mat.uniroma3.it}
\author[M.Coornaert]{Michel Coornaert}
\address{Institut de Recherche Math\'ematique Avanc\'ee \\
UMR 7501, Universit\'e de Strasbourg et CNRS \\
7 rue Ren\'e-Descartes \\
67000 Strasbourg, France  }
\email{coornaert@math.unistra.fr}
\dedicatory{Dedicated to Slava Grigorchuk on his $60$th birthday}
\date{\today}
\begin{abstract}
Let $S$ be a cancellative left-amenable semigroup and let $A$ be a finite set.
We prove that every pre-injective cellular automaton  $\tau \colon A^S \to A^S$ is surjective.
\end{abstract}
\maketitle
\section{Introduction}
\label{s:introduction}
Let $S$ be a semigroup, i.e., a set equipped with an associative binary operation.
\par
Given $s \in S$, we denote by $L_s$ and $R_s$ the left and right multiplication by $s$, that is,
the maps $L_s \colon S \to S$ and $R_s \colon S \to S$ defined by $L_s(t) = st$ and $R_s(t) = ts$ for all $t \in S$.
\par 
Let $A$ be a set,  called the \emph{alphabet}. 
 The set $A^S$, consisting of all maps $x \colon S \to A$, is called the set of 
\emph{configurations}.
Given an element $s \in S$ and a configuration $x \in A^S$, we define the configuration 
$sx \in A^S$ by $sx := x \circ R_s$. Thus, we have $sx(t) = x(ts)$ for all $ t \in S$.
The map $(s,x) \mapsto sx$ defines a left action of the semigroup $S$ on $A^S$,
that is, it satisfies $s_1(s_2x) = (s_1s_2) x$ for all $s_1,s_2 \in S$ and $x \in A^S$.
  This action is called the (left) $S$-\emph{shift} on $A^S$.
\par
We say that a map $\tau \colon A^S \to A^S$ is a \emph{cellular automaton} over the semigroup $S$ and the alphabet $A$ if   there exist a finite subset $M \subset S$ 
and a map $\mu \colon  A^M \to A$ such that 
\begin{equation} 
\label{e:local-property}
\tau(x)(s) = \mu( (s x)\vert_M)  \quad  \text{for all } x \in A^S\text{ and } s \in S,
\end{equation}
where $(sx)\vert_M \in A^M$ is the restriction of the configuration $sx = x \circ R_s$ to $M$.
Such a set $M$ is called a \emph{memory set} for $\tau$ and one says that $\mu$ is a \emph{local defining map} for $\tau$ relative to $M$. 
\par
Two configurations $x_1,x_2 \in A^S$ are said to be \emph{almost equal} if they coincide outside a finite subset of $S$.
A cellular automaton $\tau \colon A^S \to A^S$ is called \emph{pre-injective} if $\tau(x_1) = \tau(x_2)$ implies $x_1 = x_2$ whenever $x_1,x_2 \in A^S$ are almost equal.
\par
Let $\ell^\infty(S)$ denote the vector space consisting of all bounded real-valued maps 
$f \colon S \to \R$.
A \emph{mean} on $S$ is an $\R$-linear map $m \colon \ell^\infty(S) \to \R$ such that 
$\inf_{s \in S} f(s) \leq m(f) \leq \sup_{s \in S} f(s)$ for all $f \in \ell^\infty(S)$.
One says that a mean $m$ on $S$ is   \emph{left-invariant} (resp.~\emph{right-invariant}) if it satisfies
$m(f \circ L_s) = m(f)$ (resp.~$m(f \circ R_s) = m(f)$) for all
$f \in \ell^\infty(S)$ and $s \in S$.
The semigroup $S$ is called
\emph{left-amenable} (resp.~\emph{right-amenable}) if it admits a left (resp.~right) invariant mean.
One says that $S$ is \emph{amenable} if it is both left and right-amenable. 
 All commutative semigroups, all finite groups, all solvable groups, and all finitely generated groups of subexponential growth are amenable.
 For groups, it turns out that left-amenability is equivalent to right-amenability.
 Also, every subgroup of an amenable group is itself amenable.
 As non-abelian free groups are non-amenable, it follows that every group that contains a non-abelian free subgroup is non-amenable. 
On the other hand, there are semigroups that are left-amenable but not right-amenable.
There are finite semigroups that are neither left-amenable nor right-amenable and 
amenable groups containing subsemigroups that are neither left-amenable nor right-amenable.
For more on amenable groups and semigroups, see for example
 \cite{day-amenable-semigroups}, \cite{greenleaf},  \cite{namioka}, \cite{paterson}.
 \par
Fifty years ago, Moore and Myhill proved the Garden of Eden theorem for cellular automata over $\Z^2$.This theorem  states that if $A$ is a finite set and $G = \Z^2$, then 
a cellular automaton  
$\tau \colon A^G \to A^G$ is surjective if and only if it is pre-injective. 
In fact, Moore \cite{moore} first proved that surjectivity implies pre-injectivity and, shortly after, Myhill \cite{myhill} proved the converse implication.
  The Garden of Eden theorem  was subsequently extended to 
all finitely generated groups of subexponential growth in \cite{machi-mignosi} and to all amenable groups in \cite{ceccherini}.
Actually it follows from a result in   \cite{bartholdi} that the class of amenable groups is the larger class of groups for which the Moore implication holds true.
It is unknown whether the Myhill implication characterizes group amenability, i.e.,
if every non-amenable group admits a  pre-injective but not surjective cellular automaton with finite alphabet  (this is known to be true for groups containing non-abelian free subgroups). 
\par
The Garden of Eden theorem does not extend to all amenable semigroups.
For example, the additive monoid $\N$ of non-negative integers is amenable since it is commutative.
However, the shift map $\tau \colon \{0,1\}^\N \to \{0,1\}^\N$,
defined by $\tau(x)(n) = x(n + 1)$ for all $x \in \{0,1\}^\N$ and $n \in \N$,
yields an example of  a cellular automaton with finite alphabet over   $\N$ 
that is surjective but not  pre-injective (see Example \ref{ex:surjective-not-preinj} below for a generalization).  
Thus, the Moore implication is not true for the monoid $\N$.  
   \par
Recall that an element $s$ in a semigroup   $S$  is called \emph{left-cancellable} (resp. \emph{right-cancellable}) 
if the map $L_s$ (resp. $R_s$) is injective.
One says that $s$ is cancellable if it is both left-cancellable and right-cancellable.
The semigroup $S$ is called
\emph{left-cancellative} (resp. \emph{right-cancellative}, resp. \emph{cancellative}) if every element in $S$ is left-cancellable  (resp.  right-cancellable, resp. cancellable).
\par
The main result in the present paper  is that the Myhill implication remains valid for all cancellative left-amenable semigroups.
More precisely, we shall establish the following:

\begin{theorem} 
\label{t:myhill-semigroup}
Let $S$ be a cancellative left-amenable semigroup and let $A$ be a finite set.
Then every pre-injective cellular automaton $\tau \colon A^S \to A^S$ is surjective.
\end{theorem}

As injectivity implies pre-injectivity, an immediate consequence of Theorem \ref{t:myhill-semigroup} is the following result.

\begin{corollary} 
\label{c:myhill-semigroup}
Let $S$ be a cancellative left-amenable semigroup and let $A$ be a finite set.
Then every injective cellular automaton $\tau \colon A^S \to A^S$ is surjective.
\qed
\end{corollary}
Let us say that a semigroup $S$ is \emph{surjunctive} if every injective cellular automaton with finite alphabet over $S$ is surjective.
Then Corollary \ref{c:myhill-semigroup} may be rephrased by saying that every cancellative left-amenable semigroup is surjunctive.
In the final section, we will see that the bicyclic monoid is not surjunctive.
As the bicyclic monoid is amenable,  
this shows that the ipothesis of non-cancellativity cannot be removed either from Theorem \ref{t:myhill-semigroup}
or even from Corollary \ref{c:myhill-semigroup}.  
In contrast, the question whether every group is surjunctive, which
  is known as the
\emph{Gottschalk conjecture} \cite{gottschalk},
 remains open.  
However,
 it is known to be   true  for sofic groups \cite{gromov-esav}, \cite{weiss-sgds}
and the class of sofic groups is very large. It includes in particular all residually amenable groups and   
 no examples of non-sofic groups have been found up to now.
\par
The paper is organized as follows.
Sections \ref{s:interiors-boundaries}, \ref{s:folner-nets}, \ref{s:ca}, and \ref{s:tilings}
contain preliminary material on boundaries, F\o lner nets,  cellular automata, and tilings in semigroups. 
 In Section \ref{s:entropy}, given a left-cancellative semigroup $S$ and a finite set $A$, we define the entropy  of a subset of the configuration space $A^S$
with respect to a F\o lner net.  
This entropy is always bounded above by the logarithm of the cardinality of the alphabet $A$. 
Moreover, for closed invariant subsets $X \subset A^S$ and $S$ cancellative,  equality holds if and only if $X = A^S$.
Theorem \ref{t:myhill-semigroup} is established in Section \ref{s:proof}
by showing that the image of  a pre-injective cellular automaton
is a closed invariant subset of the configuration space with maximal entropy.
In  Section \ref{sec:examples},
we describe examples of cellular automata with finite alphabe over cancellative amenable semigroups that are surjective but not pre-injective. These examples generalize the shift map on $\{0,1\}^\N$ mentioned above.
  Finally, we show that the bicyclic monoid is not surjunctive.

\section{Boundaries}
\label{s:interiors-boundaries}
Let $S$ be a semigroup. 
Given two non-empty subsets  $K$ and $\Omega$  of $S$, 
we define the \emph{$K$-interior} $\inte_K(\Omega)$
 and the \emph{$K$-adherence} $\adhe_K(\Omega)$ 
of $\Omega$ by
\begin{align*}
\inte_K(\Omega) &:= \{s\in \Omega : Ks\subset \Omega \}, \\ 
\adhe_K(\Omega) &:= \{s\in S : Ks \cap\Omega \neq \varnothing\}.
\end{align*}
 Note that
$$
\inte_K(\Omega) \subset \Omega \cap \adhe_K(\Omega).
$$
We define the $K$-\emph{boundaries}
$\partial_K(\Omega)$ and $\partial^*_K(\Omega)$ 
of $\Omega$ by
$$
\partial_K(\Omega) :=\Omega\setminus \inte_K(\Omega)
\quad \text{and} \quad
\partial^*_K(\Omega)  := \adhe_K(\Omega) \setminus \inte_K(\Omega).
$$

  \begin{proposition}
Let $S$ be a semigroup.
Let $\Omega$ and $K$ be two non-empty subsets of $S$
and suppose that every element of $K$ is left-cancellable.
Then one has
\begin{equation}
\label{e:formula-for-boundary}
\partial_{K}(\Omega) =   \bigcup_{k\in K}{L_k}^{-1}(k\Omega \setminus \Omega),
\end{equation}
and
\begin{equation}
\label{e:formula-for-boundary-etoile}
\partial_{K}^*(\Omega) \subset \partial_{K}(\Omega) \coprod \left(\bigcup_{k\in K}{L_k}^{-1}(\Omega \setminus k\Omega)\right)
\end{equation}
(here $\coprod$ denotes disjoint union).
 \end{proposition}

\begin{proof}
By definition, an element $s \in S$ is in  $ \partial_K(\Omega)$ if and only if 
$s \in \Omega$
and there exists $k \in K$ such that $ks \notin \Omega$.
As $L_k$ is injective for each $k \in K$, this is equivalent to the existence of $k \in K$ such that 
$s \in L_k^{-1}(k\Omega \setminus\Omega)$. 
This shows \eqref{e:formula-for-boundary}.
\par
Let now  $s \in \partial_K^*(\Omega) = \adhe_K(\Omega) \setminus \inte_K(\Omega)$. 
If $s \in \Omega$, then we have $s \in \Omega \setminus \inte_K(\Omega) = \partial_K(\Omega)$.
Suppose now that $s \in S \setminus \Omega$.
As $s \in \adhe_K(\Omega)$,  there exists $k \in K$ such that $ks \in \Omega$.
We have $ks \notin k \Omega$ since $s \notin \Omega$ and $k$ is left-cancellable.
Thus $s \in L_k^{-1}(\Omega \setminus k\Omega)$. 
This shows \eqref{e:formula-for-boundary-etoile}. 
\end{proof}

In the sequel, we shall use  $|\cdot|$   to denote cardinality of finite sets.

\begin{corollary}
\label{c:boundary-properties-cardinalities}
Let $S$ be a semigroup. 
Suppose that $K$ and $\Omega$ are non-empty finite subsets of $S$ and that every element of $K$ is left-cancellable.
Then the sets $\partial_K(\Omega)$ and $\partial_K^*(\Omega)$ are finite. 
Moreover,  one has
\begin{equation}
\label{eq:boundary-2}
|\partial_K(\Omega)| \leq \sum_{k \in K}|k\Omega \setminus \Omega|
\quad \text{and} \quad
|\partial_K^*(\Omega)| \leq 2 \sum_{k \in K}|k\Omega \setminus \Omega|.
\end{equation} 
 \end{corollary}

\begin{proof}
  Let $k \in K$. We first observe that by left-cancellability of $k$, we have $|k\Omega| = |\Omega|$ and therefore   $ |\Omega \setminus k\Omega| =    |k\Omega \setminus \Omega|$.
  Also, the injectivity of $L_k$ implies that the sets 
$L_k^{-1}(k\Omega \setminus \Omega)$ and $L_k^{-1}(\Omega \setminus k\Omega)$  are finite of cardinality   $|L_k^{-1}(k\Omega \setminus \Omega)| = |k\Omega \setminus \Omega|$  and $|L_k^{-1}(\Omega \setminus k\Omega)| \leq  |k\Omega \setminus \Omega| = |k\Omega \setminus \Omega|$.
 Thus, taking cardinalities in 
 \eqref{e:formula-for-boundary}, we get
 \begin{align*}
|\partial_{K}(\Omega)| &= |  \bigcup_{k\in K}{L_k}^{-1}(k\Omega \setminus \Omega)| \\
& \leq \sum_{k \in K} |L_k^{-1}(k\Omega \setminus \Omega)| \\
&= \sum_{k \in K} | k\Omega \setminus \Omega|, 
\end{align*}
which yields the first inequality in \eqref{eq:boundary-2}.
On the other hand, we deduce from \eqref{e:formula-for-boundary-etoile} that
 \begin{align*}
 |\partial_{K}^*(\Omega)|
  &\leq |\partial_{K}(\Omega) \coprod \left(\bigcup_{k\in K}{L_k}^{-1}(\Omega \setminus k\Omega)\right)| \\
  &\leq |\partial_{K}(\Omega)|  + \sum_{k\in K}|{L_k}^{-1}(\Omega \setminus k\Omega)| \\
  &\leq \sum_{k \in K} | k\Omega \setminus \Omega| + \sum_{k \in K} | \Omega \setminus k\Omega| \\
&= 2 \sum_{k \in K} | k\Omega \setminus \Omega|,
 \end{align*}
which gives the  second inequality in \eqref{eq:boundary-2}.
  \end{proof}

Suppose that   $K$ and $ \Omega $ are non-empty finite subsets of a semigroup $S$ and that every element of $K$ is left-cancellable. 
We  then define the \emph{relative amenability constants} $\alpha(\Omega,K)$ and $\alpha^*(\Omega,K)$
of $\Omega$ with respect to $K$ by
$$
\alpha(\Omega,K) := \frac{\vert \partial_K(\Omega)\vert}{\vert \Omega\vert}
$$
and
$$
\alpha^*(\Omega,K) := \frac{\vert \partial_K^*(\Omega)\vert}{\vert \Omega\vert}.
$$

\section{F\o lner nets}
\label{s:folner-nets}

For left-cancellative semigroups, we have the following characterizations of left-amenability.

\begin{theorem}[F\o lner-Frey-Namioka]
\label{t:folner-equiv-amenable}
Let $S$ be a left-cancellative semigroup.
Then the following conditions are equivalent:
\begin{enumerate}[\rm (a)]
\item
$S$ is left-amenable;
\item
for every finite subset $K \subset S$ and every real number $\varepsilon > 0$, there exists a 
non-empty finite subset $F \subset S$ such that $\vert kF \setminus F \vert \leq \varepsilon  \vert F \vert$ for all $k \in K$; 
\item
there exists a directed net $(F_j)_{j \in J}$ of non-empty finite subsets of $S$ such that
\begin{equation}
\label{e:Folner-net-condition}
\lim_{j} \frac{\vert s F_j \setminus F_j \vert}{\vert F_j \vert} = 0 \quad \text{for all $s \in S$.}
\end{equation}
\end{enumerate}
\end{theorem}

\begin{proof}
The equivalence of conditions (a) and (b) follows from \cite[Corollary 4.3]{namioka}.
On the other hand, the equivalence of (b) and (c) is straightforward (see for example the discussion in   \cite[Section 1]{cck}).
\end{proof}

\begin{remark}
If we drop the left-cancellativity hypothesis in the preceding theorem, the equivalence between (b) and (c), as well as the fact that (a) implies (b), remain true.
However,  every finite semigroup $S$ trivially satisfies (b) by taking $F = S$.
As there exist finite semigroups that are not left-amenable,
it follows that the implication (b) $\Rightarrow$ (a) becomes false if we remove the 
left-cancellativity hipothesis in Theorem \ref{t:folner-equiv-amenable}.
\end{remark}

A directed net $(F_j)_{j \in J}$ of non-empty finite subsets of a semigroup $S$ satisfying 
\eqref{e:Folner-net-condition} is called a \emph{F\o lner net} for $S$. 

\begin{proposition}
\label{p:Folner-boundaries}
Let $S$ be a left-cancellative and left-amenable semigroup.
Let $(F_j)_{j \in J}$ 
be a F\o lner net for $S$ and let $K$ be a non-empty finite subset of $S$.
Then one has $\lim_j \alpha(F_j,K) = \lim_j \alpha^*(F_j,K) = 0$.
 \end{proposition}

\begin{proof}
Let   $\varepsilon > 0$.
Since $(F_j)_{j \in J}$ is a F\o lner net for $S$,
there exists $j_0 \in J$ such that
$\vert s F_j \setminus F_j \vert/\vert F_j \vert \leq \varepsilon$ for all $s \in K$ and $j \geq j_0$.
This implies 
$\alpha(F_j,K) \leq |K|\varepsilon$ and  $\alpha^*(F_j,K) \leq 2|K|\varepsilon$ for all $j \geq j_0$ by using the inequalities in Proposition~\ref{c:boundary-properties-cardinalities}.
Consequently, we have $\lim_j \alpha(F_j,K) = \lim_j \alpha^*(F_j,K) = 0$.
\end{proof}

\section{Cellular automata}
\label{s:ca}

If $E$ is a set equipped with a left action of a semigroup $S$, one says that a map $f \colon E \to E$ is $S$-\emph{equivariant} if it commutes with the $S$-action, i.e., if one has
$f(s x) = s f(x)$ for all $s \in S$ and $x \in E$.

\begin{proposition}
\label{p:ca-equivariant}
Let $S$ be a semigroup and let $A$ be a set.
Then every cellular automaton
$\tau \colon A^S \to A^S$ is $S$-equivariant.
\end{proposition}

\begin{proof}
Let $\tau \colon A^S \to A^S$ be a cellular automaton with memory set $M \subset S$
and local defining map $\mu \colon A^M \to A$.
Let $s,t \in S$ and $x \in A^S$.
 By applying \eqref{e:local-property}, we get
\[
\tau(tx)(s) = \mu( (s (t x))\vert_M)   
= \mu( ((s t) x)\vert_M) 
= \tau(x)(st) = (t \tau(x))(s).
\]
 Consequently, we have $\tau(t x) = t \tau(x)$.
This shows that $\tau$ is $S$-equivariant.
\end{proof}

\begin{proposition}
\label{p:tau-int-and-ext}
Let $S$ be a semigroup and let $A$ be a set.
Let $\tau \colon A^S \to A^S$ be a cellular automaton with memory set $M \subset S$.
Let $x_1,x_2 \in A^S$, $s \in S$,  and $\Omega \subset S$. Then the following hold:
\begin{enumerate}[\rm (i)]
\item
if the configurations $x_1$ and $x_2$ coincide on $Ms$ then   $\tau(x_1)(s)=\tau(x_2)(s)$;
  \item
if the configurations $x_1$ and $x_2$ coincide on $\Omega$ then the configurations $\tau(x_1)$ and $\tau(x_2)$ coincide on $\inte_M(\Omega)$;
\item
if the configurations $x_1$ and $x_2$ coincide outside $\Omega$ then the configurations 
$\tau(x_1)$ and $\tau(x_2)$ coincide outside $\adhe_M(\Omega)$.
\end{enumerate}
\end{proposition}

\begin{proof}
Assertion (i) immediately follows from 
formula \eqref{e:local-property}.
Assertion (i) gives us (ii) since $Ms \subset \Omega$ for all $s \in \inte_M(\Omega)$.
We also deduce (iii) from (i) since $Ms$ does not meet $\Omega$ for all $s \in S \setminus  \adhe_M(\Omega)$.
 \end{proof}

Let $S$ be a semigroup and let $A$ be a set.
We equip  the set $A^S = \prod_{s \in S} A  $  with its \emph{prodiscrete} topology, that is, with the product topology obtained by taking the discrete topology on each factor $A$ of $A^S$.
The space $A^S$ is Hausdorff and 
totally disconnected. Moreover, $A^S$ is metrizable if $S$ is countable,
and it follows from the Tychonoff product theorem that it is compact if $A$ is finite.

\begin{proposition}
\label{p:ca-continuous}
Let $S$ be a semigroup and let $A$ be a set.
Then every cellular automaton $\tau \colon A^S \to A^S$ is continuous with respect to the prodiscrete topology.
\end{proposition}

\begin{proof}
Let $\tau \colon A^S \to A^S$ be a cellular automaton with memory set $M \subset S$ and local defining map $\mu \colon A^M \to A$.
Let $x \in A^S$ and let $N \subset A^S$ be a neighborhood of $\tau(x)$.
By definition of  the prodiscrete topology on $A^S$, there exists a finite subset $\Omega \subset S$ such that $N$ contains all configurations in $ A^S$ that coincide with $\tau(x)$ on $\Omega$.
By applying Proposition~\ref{p:tau-int-and-ext}.(i),
we deduce that $\tau^{-1}(N)$ contains all configurations in $A^S$ that coincide with $x$ 
on $M\Omega = \cup_{s \in \Omega}Ms$.
Since $M\Omega$ is finite, it follows that $\tau^{-1}(N)$ is a neighborhood of $x$.
This shows that $\tau$ is continuous for the prodiscrete topology.
\end{proof}

\section{Tilings}
\label{s:tilings}
Let $S$ be a semigroup and $K$  a subset of $S$.
We say that a subset $T \subset S$ is a $K$-\emph{tiling} of $S$ if it satisfies the following conditions:
\begin{enumerate}[\rm (T-1)]
\item
if $t_1,t_2 \in T$ and $t_1 \not= t_2$ then $Kt_1 \cap Kt_2 = \varnothing$;
\item
for every $s \in S$, there exists $t \in T$ such that $Ks \cap Kt \not= \varnothing$.
\end{enumerate}

\begin{proposition}
\label{p:tilings-exist}
Let $S$ be a semigroup and let $K$ be a nonempty subset of $S$.
Then $S$ admits  a $K$-tiling. 
\end{proposition}

\begin{proof}
Consider the set $\TT$ consisting of all non-empty subsets $T \subset S$
satisfying  condition (T-1) above. The set $\TT$ is not empty since $\{s_0\} \in \TT$ for any 
$s_0 \in S$.  
On the other hand, the set $\TT$, partially ordered 
by inclusion, is inductive. 
Indeed, if $\TT'$ is a  totally ordered subset of $\TT$, then
$M = \bigcup_{T' \in \TT'} T'$ belongs to $\TT$ and is an upper bound for $\TT'$.
By applying Zorn's lemma, we deduce that $\TT$ admits a maximal element $T$. 
Then $T$ satisfies (T-1) since $T \in \TT$.
On the other hand,  $T$ also satisfies (T-2) by maximality.
Consequently, $T$ is a $K$-tiling.  
\end{proof}

\begin{proposition}
\label{p:maj--Nj-su-Fj}
Let $S$ be a left-cancellative and left-amenable semigroup. 
Let $(F_j)_{j \in J}$ be a F\o lner net for $S$. Let $K$ be a non-empty finite subset of $S$ and suppose that $T \subset S$ is a $K$-tiling of $S$.
Let us set, for each $j \in J$,
$$
T_j := \{t \in T : Kt \subset F_j\}.  
$$
Then there exist a real number $\delta > 0$ and an element $j_0 \in J$ such that  
$$
|T_j| \geq \delta | F_j | \quad \text{for all  } j \geq j_0. 
$$
\end{proposition}

\begin{proof}
Define $T_j^* \subset T$ by
\[
T_j^* := \{t \in T: Kt \cap F_j \neq \varnothing\} = T \cap \adhe_K(F_j).
\]
Consider an element $s \in \inte_K(F_j)$.
Since $T$ is a $K$-tiling, it follows from condition (T-2) that we can find $t \in T$ such that
$Ks \cap Kt \not= \varnothing$.
As $Ks \subset F_j$,
we have $t \in T_j^*$.
We deduce that
\[
\inte_K(F_j)
\subset \bigcup_{k \in K} L_k^{-1}(KT_j^*).
\]
This implies
\begin{align*}
|\inte_K(F_j)|
&\leq \left|\bigcup_{k \in K} L_k^{-1}(K T_j^*)\right| \\
&\leq \sum_{k \in K} \left| L_k^{-1}(K T_j^*)\right| \\
&\leq \sum_{k \in K} |K T_j^*| && \text{(since $L_k$ is injective for each $k \in K$)} \\
& \leq |K|^2 |T_j^*|
\end{align*}
and hence
\begin{equation}
\label{e:t-j-star}
\frac{|T_j^*|}{|F_j|} \geq \frac{\vert \inte_K(F_j) \vert}{|K|^2 |F_j|}  = 
\frac{|F_j| - |\partial_K(F_j)|}{|K|^2 |F_j|} = \frac{1 - \alpha(F_j,K)}{|K|^2}.
\end{equation}
On the other hand, 
since $T_j \subset T_j^* = T \cap \adhe_K(F_j)$ and $T \cap \inte_K(F_j) \subset T_j$,  
we have 
\[
T_j^* \setminus T_j  \subset  \adhe_K(F_j)
\setminus \inte_K(F_j)  = \partial^*_K(F_j)
\] 
and therefore
\begin{equation}
\label{e:majore-tjstar-tj}
|T_j^* \setminus T_j| \leq |\partial^*_K(F_j)|.
\end{equation}
Finally, we get
\begin{align*}
\frac{|T_j|}{|F_j|}
&=   \frac{|T_j^*|}{|F_j|} -
   \frac{|T_j^* \setminus T_j|}{|F_j|} \\
& \geq  \frac{|T_j^*|}{|F_j|} -
\frac{|\partial^*_K(F_j)|}{|F_j|} && \text{(by \eqref{e:majore-tjstar-tj})} \\
&= \frac{|T_j^*|}{|F_j|} - \alpha^*(F_j,K) \\
& \geq \frac{1 - \alpha(F_j,K)}{|K|^2}  - \alpha^*(F_j,K) 
&& \text{(by \eqref{e:t-j-star})}.
\end{align*}
By virtue of Proposition \ref{p:Folner-boundaries}, we have $\lim_j \alpha(F_j,K) = \lim_j \alpha^*(F_j,K) = 0$. Therefore we can find $j_0 \in J$ such that
$\alpha(F_j,K) \leq 1/2$ and $\alpha^*(F_j,K) \leq 1/(4|K|^2)$ for all $j \geq j_0$.
Setting $\delta = 1/(4|K|^2)$, we then get
\[
|T_j| \geq \delta |F_j|
\]
for all $j \geq j_0$.
\end{proof} 

\section{Entropy}
\label{s:entropy}

Let $S$ be a left-cancellative and left-amenable semigroup. 
Let $A$ be a finite set.
For $\Omega \subset S$, we denote by $\pi_{\Omega} \colon A^S \to A^\Omega$ the restriction map, i.e., the map defined by
$\pi_\Omega(x) = x\vert_\Omega$ for all $x \in A^S$.
Let $\FF = (F_j)_{j \in J}$ be a F\o lner net for $S$.
We define the \emph{entropy} $\ent_\FF(X)$  of a subset $X \subset A^S$ by   
\begin{equation}
\label{e:entropy}
\ent_\FF(X) := \limsup_j \frac{\log |  \pi_{F_j}(X) |}{\vert F_j \vert}.
\end{equation}
Note that one has $\ent_\FF(X) \leq \log \vert A \vert = \ent_\FF(A^S)$ and that $\ent_\FF(X) \leq \ent_\FF(Y)$ if $X \subset Y \subset A^S$. 

\begin{remark}
When $S$ is a cancellative left-amenable semigroup and $X \subset A^S$ is $S$-invariant (i.e.
$sx \in X$ for all $s \in S$ and $x \in X$), it immediately follows from the version of the Ornstein-Weiss lemma given in \cite[Theorem~1.1]{cck} that
the $\limsup$ in the definition of $\ent_\FF(X)$ is a true limit which is  independent of the choice of the F\o lner net $\FF$.
However, we will not use this fact in the proof of our main result.
\end{remark}

A fundamental property of entropy is that the entropy of a set of configurations cannot be increased by a cellular automaton. More precisely, we have the following result. 

\begin{proposition}
\label{p:tau-dcrease-ent}
Let $S$ be a left-cancellative and left-amenable semigroup,   
$\FF = (F_j)_{j \in J}$ a F\o lner net for $S$, and $A$ a finite set. 
Let $\tau \colon A^S \to A^S$ be a cellular automaton and let $X \subset A^S$.
Then one has
\[
\ent_{\FF}(\tau(X)) \leq \ent_\FF(X).
\]
\end{proposition}

\begin{proof}
Let $Y := \tau(X) \subset A^S$ denote the image of $X$ by $\tau$.
Suppose that $M \subset S$ is a memory set for $\tau$.
By Proposition~\ref{p:tau-int-and-ext}.(ii),
if  two configurations $x_1,x_2 \in X$ coincide on $F_j$ then
$\tau(x_1)$ and $\tau(x_2)$ coincide on $\inte_M(F_j)$.
It follows that
\begin{equation}
\label{e:bound-tau-Y-int}
|\pi_{\inte_M(F_j)}(Y)| \leq |\pi_{F_j}(X)|.
\end{equation}
On the other hand, as $F_j$ is the disjoint union of $\inte_M(F_j)$ and $\partial_M(F_j)$, we have
$$
\pi_{F_j}(Y) \subset \pi_{\inte_M(F_j)}(Y) \times A^{\partial_M(F_j)}
$$
and hence
\[
\log |\pi_{F_j}(Y)| \leq
\log |\pi_{\inte_M(F_j)}(Y)| + |\partial_M(F_j)| \log |A|.  
\]
After dividing by $|F_j|$, this gives us
\begin{align*}
\frac{\log \vert \pi_{F_j}(Y) \vert}{\vert F_j \vert}
&\leq \frac{\log \vert \pi_{\inte_M(F_j)}(Y) \vert}{\vert F_j \vert} + \alpha(F_j,M) \log |A| \\
&\leq \frac{\log \vert \pi_{F_j}(X) \vert}{\vert F_j \vert} + \alpha(F_j,M) \log \vert A \vert 
&& \text{(by \eqref{e:bound-tau-Y-int}).}
 \end{align*}
  Since $\lim_j \alpha(F_j,M) = 0$ by Proposition \ref{p:Folner-boundaries},
 we finally get
 $$
\ent_\FF(Y) = \limsup_j \frac{\log \vert \pi_{F_j}(Y) \vert}{\vert F_j \vert}
\leq \limsup_j \frac{\log \vert \pi_{F_j}(X) \vert}{\vert F_j \vert} = \ent_\FF(X).
\vspace*{-6pt}
$$
 \end{proof}

The following result may be used to show that certain sets of configurations
do not have maximal entropy.
  
\begin{proposition}
\label{p:ent-strict-inf}
Let $S$ be a left-cancellative and left-amenable semigroup,   
$\FF = (F_j)_{j \in J}$ a F\o lner net for $S$, and $A$ a finite set. 
Suppose that a subset $X \subset A^S$ satisfies the following condition:
there exist a non-empty finite subset $K \subset S$ and a $K$-tiling $T \subset S$ of $S$ such that
$\pi_{Kt}(X) \subsetneqq A^{Kt}$ for all $t \in T$. 
Then one has $\ent_\FF(X) < \log\vert A \vert$.
\end{proposition}

\begin{proof}
For each $j \in J$, consider   
the subset $T_j \subset T$ defined by
$T_j :=   \{t \in T : Kt \subset F_j\}$ and the subset $F_j^* \subset F_j$ given by
$$
F_j^* = F_j \setminus \coprod_{t \in T_j} Kt.
$$
As the sets $F_j^*$ and $Kt$, $t \in T_j$, form a partition of $F_j$, we have
\begin{equation}
\label{e:card-Fj-F_jstar-Kt}
\vert F_j \vert = \vert F_j^* \vert + \sum_{t \in T_j} \vert Kt \vert.
\end{equation}
On the other hand, by our hipothesis, we have
\begin{equation}
\label{e:maj-pi-X-Nj}
\vert \pi_{Kt}(X) \vert \leq \vert A^{Kt} \vert - 1 = \vert A \vert^{\vert Kt \vert} - 1
\quad \text{for all } t \in T.
\end{equation}
As
$$
\pi_{F_j}(X) \subset A^{F_j^*} \times \prod_{t \in T_j} \pi_{Kt}(X),
$$
it follows that
\begin{align*}
\log \vert \pi_{F_j}(X) \vert 
&\leq \log \vert A^{F_j^*} \times \prod_{t\in T_j} \pi_{Kt}(X) \vert \\
&= \vert F_j^* \vert \log \vert A \vert + \sum_{t \in T_j} \log \vert \pi_{Kt}(X) \vert \\
&\leq \vert F_j^* \vert \log \vert A \vert + \sum_{t \in T_j} \log(\vert A \vert^{\vert Kt \vert} - 1) \ \ \text{(by \eqref{e:maj-pi-X-Nj})}\\
&= \vert F_j^* \vert \log \vert A \vert + \sum_{t \in T_j} \vert Kt \vert \log\vert A \vert
+ \sum_{t \in T_j} \log(1 - \vert A \vert^{-\vert Kt \vert} ) \\
&= |F_j| \log |A|
+ \sum_{t \in T_j} \log(1 - \vert A \vert^{-\vert Kt \vert})  &&\text{(by \eqref{e:card-Fj-F_jstar-Kt})}\\
&\leq  \vert F_j \vert \log\vert A \vert
+ \vert T_j \vert \log(1 - \vert A \vert^{-\vert K \vert})\\
  &\text{(since $|Kt| \leq |K|$ for all $t \in T_j$).} 
\end{align*}

By introducing the constant  $c := - \log(1 - \vert A \vert^{- \vert K \vert}) > 0$,   this gives us
$$
\log\vert \pi_{F_j}(X) \vert \leq \vert F_j \vert \log\vert A \vert - c \vert T_j \vert \quad 
\text{ for all } j \in J.
$$
Now, by Proposition \ref{p:maj--Nj-su-Fj}, there exist $\delta > 0$ and $j_0 \in J$ such that $\vert T_j \vert \geq \delta
\vert F_j \vert$ for all $j \geq j_0$. Thus
$$
\frac{\log\vert \pi_{F_j}(X) \vert}{\vert F_j \vert} \leq \log\vert A \vert - c\delta \quad \text{ for all } j \geq j_0.
$$
This implies that
$$
\ent_\FF(X) = \limsup_j \frac{\log\vert \pi_{F_j}(X) \vert}{\vert F_j \vert} \leq \log\vert A \vert - c\delta < \log\vert A \vert.
$$
\end{proof}

Recall that if $E$ is a set equipped with a left action of a semigroup $S$, one says that a subset $X \subset E$ is $S$-\emph{invariant} if one has $sx \in X$ for all $s \in S$ and $x \in X$.

\begin{corollary}
\label{c:entropy-invariant-not-max}
Let $S$ be a cancellative and left-amenable semigroup,   
$\FF = (F_j)_{j \in J}$ a F\o lner net for $S$, and $A$ a finite set. 
Suppose that $X \subset A^S$ is an $S$-invariant subset satisfying the following condition:
there exists a  finite subset $K \subset S$ such that
\begin{equation}
\label{e:proj-on-K-not-onto}
\pi_{K}(X) \subsetneqq A^{K}. 
\end{equation}
Then one has $\ent_\FF(X) < \log\vert A \vert$.
\end{corollary}

\begin{proof}
First observe that
\begin{equation}
\label{e:proj-on-Ks-not-onto}
\pi_{Ks}(X) \subsetneqq A^{Ks} \quad \text{for all  } s \in S. 
\end{equation}
Indeed, by \eqref{e:proj-on-K-not-onto}, we can find $u \in A^K$ such that $u \notin \pi_{K}(X)$.
Then, given $s \in S$, we can define $v \in A^{Ks}$ by setting $v(ts) = u(t)$ for all $t \in K$ (the right-cancellability of $s$ implies that $v$ is well defined).
Now, there is no $x \in X$ such that $\pi_{Ks}(x) = v$ since otherwise the configuration $sx$, which is in $X$ by our $S$-invariance hipothesis, would satisfy $\pi_{K}(sx) = u$. This proves \eqref{e:proj-on-Ks-not-onto}. 
\par
Since we can find a $K$-tiling of $S$ by Proposition~\ref{p:tilings-exist}, 
we deduce from \eqref{e:proj-on-Ks-not-onto} that $\ent_\FF(X) < \log\vert A \vert$ by applying Proposition~\ref{p:ent-strict-inf}. 
\end{proof}

\begin{corollary}
\label{c:entropy-invariant-closed-max}
Let $S$ be a cancellative and left-amenable semigroup,   
$\FF = (F_j)_{j \in J}$ a F\o lner net for $S$, and $A$ a finite set. 
Suppose that $X \subset A^S$ is a closed (for the prodiscrete topology) 
and $S$-invariant subset of $A^S$.
Then one has $\ent_\FF(X) = \log |A|$ if and only if $X = A^G$.
\end{corollary}

\begin{proof}
The fact that $\ent_\FF(A^G) = \log |A|$ has already been observed and is trivial.
Conversely, suppose that $X \subsetneqq A^G$.
As $X$ is closed in $A^S$, this means that we can find a subset $K \subset S$ such that  
$\pi_{K}(X) \subsetneqq A^{K}$.
It then follows from Corollary \ref{c:entropy-invariant-not-max} that $\ent_\FF(X) < \log |A|$.
\end{proof}
 
\section{Entropy and cellular automata}
\label{s:proof} 

In this section, we give the proof of Theorem \ref{t:myhill-semigroup}.
We shall use the following auxiliary result.

\begin{lemma}
\label{l:preinj-implique-ent-max}
Let $S$ be a left-cancellative and left-amenable semigroup,   
$\FF = (F_j)_{j \in J}$ a F\o lner net for $S$, and $A$ a finite set. 
 Suppose that $\tau \colon A^S \to A^S$ is a cellular automaton such that
\begin{equation} \label{e:ent-imae-non-max}
\ent_{\FF}(\tau(A^S)) < \log\vert A \vert.
\end{equation}
Then $\tau$ is not pre-injective.
\end{lemma}

\begin{proof}
Let $Y := \tau(A^S)$ denote the image of $\tau$ and let $M \subset S$ be a memory set for $\tau$.
 Recall that $\adhe_M(F_{j})$ is the disjoint union of
$\inte_M(F_j)$ and $\partial^*_{M}(F_j)$. Therefore we have 
\[
\pi_{\adhe_M(F_{j})}(Y) \subset \pi_{\inte_M(F_j)}(Y) \times A^{\partial^*_{M}(F_j)}.
\]
This implies
 \begin{align*}
\log \vert \pi_{\adhe_M(F_{j})}(Y) \vert &\leq \log\vert \pi_{\inte_M(F_j)}(Y) \vert + \vert \partial^*_{M}(F_j)\vert
\log \vert A \vert \\
&\leq \log\vert \pi_{F_j}(Y) \vert + \vert \partial^*_{M}(F_j) \vert \log \vert A \vert && \text{(since $\inte_M(F_j) \subset F_j$)}. 
\end{align*}
After dividing by $|F_j|$, we get
\begin{equation}
\label{e:maj-pi-j-pro-Y}
\frac{\log \vert \pi_{\adhe_M(F_{j})}(Y) \vert }{\vert F_j \vert} \leq \frac{\log\vert \pi_{F_j}(Y) \vert}{\vert F_j \vert} +
\alpha^*(F_j,M)
\log\vert A \vert
\end{equation}
for all $j \in J$.
As
$$
  \limsup_j \frac{\log\vert \pi_{F_j}(Y) \vert}{\vert F_j \vert} = \ent_\FF(Y) < \log \vert A \vert
$$
by our hipothesis, and
$$
\lim_j \alpha^*(F_j,M) = 0
$$
by Proposition \ref{p:Folner-boundaries}, we deduce from inequality \eqref{e:maj-pi-j-pro-Y} that
$$
\limsup_j \frac{\log \vert \pi_{\adhe_M(F_{j})}(Y) \vert }{\vert F_{j} \vert} < \log \vert A \vert. 
$$ 
Consequently, there exists $j_0 \in J$ such that
\begin{equation} 
\label{e:maj-pi-Y-j0}
\frac{\log \vert \pi_{\adhe_M(F_{j_0})}(Y) \vert }{\vert F_{j_0} \vert} < \log \vert A \vert.
\end{equation}
Now let us fix an arbitrary element $a_0 \in A$ and consider  the set  $Z \subset A^S$ consisting of all the  configurations $z \in A^S$ such that $z(s) = a_0$ for all $s \in S \setminus F_{j_0}$.
Note that the set $Z$ is finite of cardinality
\[
|Z| = |A|^{|F_{j_0}|}.
\]
 Inequality~\eqref{e:maj-pi-Y-j0} gives us
\begin{equation}
\label{e:minore-card-Z}
\vert \pi_{\adhe_M(F_{j_0})}(Y) \vert <   \vert Z \vert.
\end{equation}
Observe that if $z_1,z_2 \in Z$, then $z_1$ and $z_2$ coincide outside $F_{j_0}$ so that the image configurations  $\tau(z_1)$ and $\tau(z_2)$ coincide outside $\adhe_M(F_{j_0})$  by Proposition \ref{p:tau-int-and-ext}.(iii). 
Thus we have
$$
\vert \tau(Z) \vert = \vert \pi_{\adhe_M(F_{j_0})}(\tau(Z)) \vert \leq \vert \pi_{\adhe_M(F_{j_0})}(Y) \vert
$$
and hence, by using \ref{e:minore-card-Z}, 
$$
|\tau(Z)| < \vert Z \vert.
$$
This last inequality implies that we can find two distinct configurations $z_1$ and $z_2$ in $Z$ such that $\tau(z_1) =  \tau(z_2)$.
Since $z_1$ and $z_2$ are almost equal (they coincide outside the finite set  
$F_{j_0}$), this shows that $\tau$ is not pre-injective.
\end{proof}

\begin{proof}[Proof of Theorem \ref{t:myhill-semigroup}]
Let   $\tau \colon A^S \to A^S$ be a   cellular automaton and suppose that $\tau$ is pre-injective.
 Choose a F\o lner net $\FF = (F_j)_{j \in J}$ for $S$.
 then the  image set $Y := \tau(X)$
 satisfies $\ent_\FF(Y) = \log |A|$ by Lemma \ref{l:preinj-implique-ent-max}.
On the other hand,   as $A^S$ is a compact Hausdorff space and $\tau$ is continuous for the prodiscrete topology by Proposition~\ref{p:ca-continuous},
the  set $Y $ is closed in $A^S$.
 Since $Y$ is $S$-invariant by Proposition~\ref{p:ca-equivariant},
it then follows from Corollary~\ref{c:entropy-invariant-closed-max} that $Y = A^S$.
This shows that $\tau$ is surjective. 
\end{proof}

\begin{remark}
\label{r:commutative}
Recall that a semigroup $S$ is said to be \emph{left-reversible} if any two left-principal ideals in $S$ intersect, i.e., $aS \cap bS \not= \varnothing$ for all $a,b \in S$.
As every left-amenable semigroup is clearly left-reversible,
one deduces from Ore's theorem that if $S$ is a cancellative left-amenable semigroup, then $S$ embeds in an amenable group, its group of left-quotients $G := \{st^{-1} : s,t \in S\}$
(see \cite[Corollary 3.6]{wilde-witz}).
When $S$ is a cancellative commutative semigroup, e.g., $S = \N$ for which $G = \Z$,
given any finite subset $F \subset G$, we can always find $t \in S$ such that $t + F \subset S$
(if $F = \{s_i - t_i  : s_i,t_i \in S, 1 \leq i \leq n\}$, we can take $t = \sum_{1 \leq i \leq n} t_i$).
It follows that  the Myhill property for cellular automata over $S$ may be easily deduced from the Myhilll property for cellular automata over $G$
in that particular case.    
Indeed, suppose that $\tau \colon A^S \to A^S$ is a cellular automaton with memory set $M \subset S$ and local defining map
$\mu \colon A^M \to A$.
Consider the cellular automaton $\sigma \colon A^G \to A^G$ that admits $M$ as a memory set and $\mu$ as a local defining map.
If two configurations $x_1,x_2 \in A^G$ coincide outside $F$, then
their shifts by $-t$ coincide outside $t + F \subset S$.
We deduce that the pre-injectivity of $\tau$ implies the pre-injectivity of $\sigma$.
As the surjectivity of $\sigma$ clearly implies the surjectivity of $\tau$, this proves our claim.   
\end{remark}

\section{Some examples of cellular automata}
\label{sec:examples}
 
\begin{example}[Surjective but not pre-injective cellular automata]
\label{ex:surjective-not-preinj}
Let $S$ be a semigroup admitting a left-cancellable element $s_0$ such that $s_0 S \not= S$
(i.e., an element $s_0$ such that the left-multiplication map $L_{s_0} \colon S \to S$ is injective but not surjective). 
 Take $A = \{0,1\}$,
and consider the map $\tau \colon A^S \to A^S$
defined by $\tau(x)(s) = x(s_0 s)$ for all $x \in A^S$ and $s \in S$.
Clearly $\tau$ is a cellular automaton over the semigroup $S$ admitting $M = \{s_0\}$ as a 
memory set.
Let $x_0 \in A^S$ be the configuration defined by $x_0(s) = 0$ for all $s \in S$.
Choose an arbitrary element $s_1 \in S \setminus s_0 S$ and  
let $x_1 \in A^S$ be the configuration defined by  $x_1(s) = 0$ for all $s \not=  s_1$ and $x_1(s_1) = 1$.
Then we have $x_0 \not= x_1$ but $\tau(x_0) = \tau(x_1) = x_0$.
As the configurations $x_0$ and $x_1$ are almost equal, this shows that $\tau$ is not pre-injective. 
On the other hand, let $y \in A^S$.
Consider the configuration $x \in A^S$ defined by $x(s) = 0$ if $s \notin s_0 S$ and
$x(s) = y(t)$ if $s = s_0 t$ for some $t \in S$
(the left-cancellability of $s_0$ guarantees that $x$ is well defined). 
Then we  have $\tau(x) = y$.
Consequently, $\tau$ is surjective.
Any free semigroup, any free monoid, any free commutative semigroup, and any free commutative monoid satisfies our hipothesis on $S$ as soon as it is non-trivial.
Note that all free commutative semigroups and all free commutative monoids are cancellative and amenable.
\end{example}

\begin{example}[Non-surjunctivity of the bicyclic monoid]
We recall that the \emph{bicyclic monoid} is the monoid $B$ with presentation $B = \langle p,q : pq = 1\rangle$ and
that every element $s \in B$ can be uniquely written in the form
$s = q^a p^b$, where $a = a(s)$ and $b = b(s)$ are  non-negative integers.
It is known (see for example \cite[Example 2, page 311]{duncan-namioka}) that the bicyclic monoid is an amenable inverse semigroup.
\par
 Take $A = \{0,1\}$,
and consider the map $\tau \colon A^B \to A^B$
defined by $\tau(x)(s) = x(p s)$ for all $x \in A^B$ and $s \in B$.
Clearly $\tau$ is a cellular automaton over  $B$ admitting $M = \{p\}$ as a 
memory set.
Observe that $\tau$ is not surjective since $1_B \not= qp$ and $\tau(x)(1_B) = \tau(x)(qp)$ for all $x \in A^B$.
On the other hand, $\tau$ is injective since $pB = B$.
\end{example}

\begin{remark}
It would be interesting to give an example of a cancellative semigroup that is not surjunctive.
\end{remark}

\noindent
\emph{Acknowledgments.} We  are grateful to the referee for helpful suggestions.\\

\noindent
\emph{Note added in proof.} We thank Laurent Bartholdi who pointed out to us that our argument in Remark \ref{r:commutative} can be extended to all cancellative left-amenable semigroups, thus yielding an alternative proof of Theorem \ref{t:myhill-semigroup}. Suppose indeed that $S$ is a cancellative left-amenable semigroup and denote by $G := \{st^{-1} : s,t \in S\}$ the amenable group of its left-quotients. Let $F = \{s_it_i^{-1}: s_i, t_i \in S, 1 \leq i \leq n\} \subset G$ be a finite subset. Since $S$ is left-reversible, we have $\bigcap_{i=1}^n t_iS \neq \varnothing$.
Taking $t \in \bigcap_{i=1}^n t_iS$ gives $Ft \subset S$ and the remaining arguments (based on the Myhilll property for cellular automata over $G$) in Remark \ref{r:commutative}
apply verbatim.


\end{document}